\documentclass[11pt, a4paper]{article}

\usepackage{amsmath,amsthm,amssymb}
\usepackage{tikz}
\usepackage{fullpage}

\theoremstyle{plain}

\newtheorem{corollary}{Corollary}
\newtheorem{lemma}{Lemma}
\newtheorem{proposition}{Proposition}
\newtheorem{theorem}{Theorem}
\newtheorem*{claim*}{Claim}

\theoremstyle{definition}
\newtheorem{definition}{Definition}

\begin{document}

\title{New constructions and bounds for Winkler's hat game}
\author{Maximilien Gadouleau\footnote{School of Engineering and Computing Sciences, Durham University, Durham, UK. Email: \texttt{m.r.gadouleau@durham.ac.uk}}\and Nicholas Georgiou\footnote{Department of Mathematical Sciences, Durham University, Durham, UK. Email: \texttt{nicholas.georgiou@durham.ac.uk}}}

\maketitle

\begin{abstract}
Hat problems have recently become a popular topic in combinatorics and discrete mathematics. These have been shown to be strongly related to coding theory, network coding, and auctions. We consider the following version of the hat game, introduced by Winkler and studied by Butler et al. A team is composed of several players; each player is assigned a hat of a given colour; they do not see their own colour, but can see some other hats, according to a directed graph. The team wins if they have a strategy such that, for any possible assignment of colours to their hats, at least one player guesses their own hat colour correctly. In this paper, we discover some new classes of graphs which allow a winning strategy, thus answering some of the open questions in Butler et al. We also derive upper bounds on the maximal number of possible hat colours that allow for a winning strategy for a given graph.
\end{abstract}

\section{Introduction} \label{sec:introduction}

Hat games are a popular topic in combinatorics. Typically, a hat game involves $n$ players, each wearing a hat that can take a colour from a given set of $q$ colours. No player can see their own hat, but each player can see some subset of the other hats. All players are asked to guess the colour of their own hat at the same time. For an extensive review of different hat games, see \cite{Krz12}. Different variations have been proposed: for instance, the players can be allowed to pass \cite{Ebe98}, or the players can guess their respective hat's colour sequentially \cite{Krz10}. The variation in \cite{Ebe98} mentioned above has been investigated further (see \cite{Krz12}) for it is strongly connected to coding theory via the concept of covering codes \cite{CHLL97}; in particular, some optimal solutions for that variation involve the well-known Hamming codes \cite{EMV03}. In the variation called the ``guessing game,'' players are not allowed to pass, and must guess simultaneously \cite{Rii07}. The team wins if everyone has guessed their colour correctly; the aim is to maximise the number of hat assignments which are correctly guessed by all players. This version of the hat game has been further studied in \cite{Rii07a, GR11} due to its relations to graph entropy, to circuit complexity, and to network coding, which is a means to transmit data through a network which allows the intermediate nodes to combine the packets they receive \cite{ACLY00}.

In this paper, we are interested in the following hat problem, a small variation to Winkler's hat game presented in \cite{Win01}. We are given a directed graph $D$ (without loops and repeated arcs, but possibly with bidirectional edges) on $n$ vertices and a finite alphabet $[q] = \{0,\ldots,q-1\}$ ($q \ge 2$).  We say that $f = (f_1,\ldots,f_n) : [q]^n \to [q]^n$ is a $D$-function if every local function $f_v: [q]^n \to [q]$ only depends on the values in the in-neighbourhood of $v$ in $D$: $f_v(x) = f_v(x_{N^-(x)})$. We ask whether there is a $D$-function over $[q]$ such that for any $x = (x_1,\ldots,x_n) \in [q]^n$, $f_v(x) = x_v$ for some vertex $v$. In that case, we say that $D$ is $q$-{\em solvable} and that $f$ \emph{solves} $D$.

In terms of the hat game, each vertex in the graph represents a player, an arc from player $u$ to $v$ means that $v$ can see $u$. The set $[q]$ then represents the possible colours of their hats and $x = (x_1,\ldots,x_n) \in [q]^n$ represents a possible hat assignment. Each player $v$ must guess the colour of their hat according to some pre-determined rule which can only depend on the hats that they see: $f_v(x_{N^-(x)})$. If one player guesses correctly, i.e., $x_v = f_v(x)$, then the team wins; if all guess incorrectly, the team loses. The question is then to come up with a winning strategy regardless of the hat assignment.

Clearly, if $D$ is $q$-solvable, then it is also $(q-1)$-solvable. The clique $K_q$ is $q$-solvable \cite{Win01}: if we denote the players as elements in $[q]$, then $v$ guesses that the sum of all hat assignments is equal to $v$ modulo $q$: $f_v(x) = -\sum_{u \ne v} x_u + v$. More generally, if the players play on $K_n$, then there is a strategy which guarantees that at least $\lfloor n/q \rfloor$ players guess correctly (simply split $K_n$ into $\lfloor n/q \rfloor$ cliques $K_q$). The case for $K_n$ and $q=2$ colours with unequal probabilities was further studied in \cite{Fei04, Doe05}; its relation to auctions has been revealed in \cite{AFGHIS11} and developed in \cite{BNW13}. 

Results for other classes of graphs have been found in the literature. Butler et al.\@ proved in \cite{BHKL08} that for any $q$, there exists a $q$-solvable undirected bipartite graph. Unfortunately, that graph has a doubly exponential number of vertices. In the same paper, they also proved that undirected trees are not $3$-solvable. 

The main contributions of this paper are as follows. In \cite{BHKL08}, it is asked whether there exist $K_q$-free $q$-solvable undirected graphs with a polynomial number (in $q$) of vertices. We give an emphatic affirmative answer: for any $\epsilon$, there exist $K_{\epsilon q}$-free $q$-solvable graphs with a linear number of vertices; moreover, we present a class of $K_\omega$-free graphs with $\omega = o(q)$ which are $q$-solvable and have a polynomial number of vertices. We also refine the multiplicative constant for some values of $\epsilon$ by considering small undirected graphs or directed graphs. We also prove some non-solvability results for bipartite graphs and for graphs with a large independent set. Another question asked in \cite{BHKL08} concerns so-called \emph{edge-critical} graphs, i.e., undirected graphs which are $q$-solvable but which have no $q$-solvable proper spanning subgraph. Clearly, the only edge-critical graph for $q=2$ colours is $K_2$; \cite{BHKL08} asks whether there exists an infinite family of edge-critical graphs for any other $q \ge 3$. By studying the solvability of cycles, we are able to show that the cycles whose length are a multiple of six form an infinite family of edge-critical graphs for 3 colours.

The rest of the paper is organised as follows. In Section~\ref{sec:undirected}, we prove the existence of bipartite or $K_\omega$-free $q$-solvable undirected graphs with a relatively small number of vertices. In Section~\ref{sec:directed}, we refine some constructions by extending our consideration to directed graphs. We then derive some non-solvability results in Section~\ref{sec:non-solvability}. Finally, we prove the existence of a class of edge-critical $3$-solvable graphs in Section~\ref{sec:even_cycles}.


\section{Undirected constructions} \label{sec:undirected}

In \cite{BHKL08}, it is proved that for any $q \ge 2$ there exists a $q$-solvable bipartite graph with a doubly exponential number of vertices ($q^{q^{q-1}} + q-1$ vertices to be exact). We refine their argument to construct a $q$-solvable bipartite graph with only an exponential number of vertices.

%
%
%

We say that a set of words $S$ in $[q]^m$ is {\em distinguishable} if there exists a word $x \in [q]^m$ such that $d_H(x,s) \le m-1$ for all $s \in S$, where $d_H$ is the Hamming distance.  Alternatively, using the terminology of \cite{CG12}, this is equivalent to $S$ having remoteness at most $m-1$. The main reason we are interested in distinguishable sets is as follows. If in a graph there is an independent set $M$ of cardinality $m$, and the vertices in $M$ know that their hat assignment $x \in [q]^m$ is any possible element of a set $S \subseteq [q]^m$, then there exist guessing functions for the vertices of $M$ achieving at least one correct guess if and only if $S$ is distinguishable. 

\begin{theorem}[See \cite{BHKL08}] \label{th:bipartite}
The complete bipartite graph $K_{q-1,(q-1)^{q-1}}$ is $q$-solvable.
\end{theorem}

\begin{proof}
Set $m=q-1$, and label the left vertices of $K_{q-1,(q-1)^{q-1}}$ by $v_1,\dotsc,v_m$.  Write $[q]_+$ for the set $\{1,\dots,q-1\}$ (so $[q]_+ \subseteq [q]$) and label the right vertices of $K_{q-1,(q-1)^{q-1}}$ by $w_z$ for $z \in [q]_+^m$.  For each $z \in [q]_+^m$ define the guessing function $f_z : [q]^m \to [q]$ by
\[
f_z(x) = \begin{cases}
0 & \text{if $d_H(x,z) = m$}\\
\min\{ i : x_i = z_i \} & \text{if $d_H(x,z) < m$}\\
\end{cases}
\]

It is enough to show that for any hat configuration $(x,y) = (x_1,\dots,x_m,y_{(1,\dots,1)},\dots,y_{(q-1,\dots,q-1)})$ if all the vertices $w_z$ guess incorrectly, then the vertices $v_i$ know that the vector $x$ lies in some distinguishable set.

That is, it is enough to show that for all $y$ there exists $a \in [q]^m$ such that
\[
\bigcap_{z \in [q]_+^m} f_z^{-1}(y_z)^{\rm c} \subseteq B_{m-1}(a).
\]
(The $m$ components of the vector $a$, which depends on $y$, are exactly the guessing functions for the vertices $v_1,\dots,v_m$.)

We prove by (reverse) induction on $i$ the following:
\begin{claim*}
Suppose $(x,y) \in [q]^m \times [q]^{[q]_+^m}$ is a configuration of hats guessed incorrectly by every vertex.  Then, for every $i = 1,\dots,m$, and every $(z_1,\dots,z_{i-1}) \in [q]_+^{i-1}$ there exists $(z_i,\dots,z_m) \in [q]_+^{m-i+1}$ with $y_{(z_1,\dots,z_m)} \not\in \{i,\dotsc,m\}$.
\end{claim*}

\begin{proof}[Proof of Claim]
Let $i=m$, and fix $z_1,\dots,z_{m-1}$.  Consider the variables $y_{(z_1,\dots,z_{m-1},z)}$ for $z \in [q]_+$; if all are equal to $m$, then
\[
X_m(z) := f_{(z_1,\dots,z_{m-1},z)}^{-1}(y_{(z_1,\dots,z_{m-1},z)}) = \{ x \in [q]^m : x_i \neq z_i \text{ for all $i < m$ and } x_m = z\}.
\]
Hence
\[
\bigcup_{z \in [q]_+} X_m(z) = \{ x \in [q]^m : x_i \neq z_i \text{ for all $i<m$ and } x_m \neq 0 \}
\]
implying that $\displaystyle\bigcap_{z \in [q]_+} X_m(z)^{\rm c} = B_{m-1}(z_1,\dotsc,z_{m-1},0)$, contradicting the fact that the vertices $v_1,\dotsc,v_m$ guess incorrectly.  Therefore there exists some $z \in [q]_+$ with $y_{(z_1,\dots,z_{m-1},z)} \neq m$.

Now, suppose the statement is true for $i >1$; we show it holds for $i-1$.  Fix $z_1,\dots,z_{i-2}$; for each $a \in [q]_+$, by our inductive hypothesis there exist $z_i(a),\dots,z_m(a) \in [q]_+$ with
\[
y_{(z_1,\dotsc,z_{i-2},a,z_i(a),\dots,z_m(a))} \not \in \{i,\dots,m\}.
\]
So, it is enough to show that for at least one $a \in [q]_+$ the variable $y_{(z_1,\dotsc,z_{i-2},a,z_i(a),\dots,z_m(a))}$ is not equal to $i-1$.  For a contradiction, suppose not, so that all such variables equal $i-1$.  Then,
\[
\begin{split}
X_{i-1}(a) &:= f_{(z_1,\dots,z_{i-2},a,z_i(a),\dots,z_m(a))}^{-1}(y_{(z_1,\dots,z_{i-2},a,z_i(a),\dots,z_m(a))})\\
&= \{ x \in [q]^m : x_j \neq z_j \text{ for all $j < i-1$ and } x_{i-1} = a\}.
\end{split}
\]
Therefore,
\[
\bigcup_{a \in [q]_+} X_{i-1}(a) = \{ x \in [q]^m : x_j \neq z_j \text{ for all $j<i-1$ and } x_{i-1} \neq 0 \}
\]
implying that $\displaystyle\bigcap_{a \in [q]_+} X_{i-1}(a)^{\rm c} \subseteq B_{m-1}(z_1,\dotsc,z_{i-2},0,\dots,0)$ contradicting the fact that $v_1,\dots,v_m$ guess incorrectly.
\end{proof}

Finally, applying the claim for $i=1$, we find a $z \in [q]_+^m$ where $y_z$ cannot take any value in $\{1,\dots,m\}$.  This implies that $y_z = 0$ and $f_z^{-1}(y_z)^{\rm c} = B_{m-1}(z)$, so that at least one of $v_1,\dots,v_m$ guesses correctly.
\end{proof}

The {\em lexicographic product} of a directed graph $D = (V,E)$ and a clique $K_r$, denoted as $(D,r)$, is defined as the graph with vertex set $V \times [r]$, where $((u,a), (v,b))$ is an arc if and only if either $(u,v) \in E$ or $u=v$ and $a\neq b$. If $D$ has $n$ vertices and clique number $\omega$, then the graph $(D,r)$ has $rn$ vertices and clique number $r \omega$.

\begin{lemma}[The blow-up lemma] \label{lem:blow-up}
If $G$ is a $p$-solvable directed graph, then $(G, r)$ is a $q$-solvable graph, where $q = pr$.
\end{lemma}

\begin{proof}
Let $f$ be the corresponding guessing function that solves $G$ over $p$ colours. For any vertex $(v,a)$ in $(G,r)$, we denote the configuration as $(x_{(v,a)}, y_{(v,a)}) \in [p] \times [r]$ and we also denote $X_v = \sum_{a \in [r]} x_{(v,a)}$, $Y_v = \sum_{a \in [r]} y_{(v,a)}$ and write $X$ for the vector $(X_v,v \in G)$. We claim that the $(G,r)$-function $g$, defined as follows for each $(v,a)$, never fails:
\[
	g_{(v,a)}(x,y) = \left(f_v(X) - X_v + x_{(v,a)}, -Y_v + y_{(v,a)} - a \right).
\]
Suppose $(x,y)$ is guessed wrong by all vertices. In particular, it is guessed incorrectly by $(v,a)$, hence either $f_v(X) \ne X_v$ or $Y_v \ne a$. Since this holds for all $a$, in particular this holds for $a = Y_v$; we conclude that $f_v(X) \ne X_v$. Since this holds for all $v$, this violates the fact that $f$ is a solution for $G$.
\end{proof}

\begin{theorem} \label{th:epsilon}
For any $\epsilon > 0$, there exists $n_\epsilon$ such that the following holds. For any $q$, there exists a $q$-solvable undirected graph with at most $n_\epsilon q$ vertices and clique number $\epsilon q$.
\end{theorem}

\begin{proof}
Firstly, let $p = \lfloor 1/\epsilon \rfloor + 1$ and let $q$ be divisible by $p$. Let $G_p$ be the $p$-solvable bipartite graph in Theorem \ref{th:bipartite} and let $g_p$ denote its size. Then by the blow-up lemma, $(G_p,q/p)$ is a $q$-solvable graph with $g_p q/p$ vertices and clique number $2q/p$. If $q$ is not divisible by $p$, consider $q' = p \lceil q/p \rceil \le q(1 + 1/p)$ and $n_\epsilon = (1 + 1/p)g_p/p$.
\end{proof}

\begin{theorem} \label{th:K-free}
For any $\omega$ such that $\omega \ge \frac{q}{m} \frac{\log \log q}{\log q}$ holds for large enough $q$ and some $m > 0$, there exists a $q$-solvable $K_\omega$-free undirected graph with at most $q^{2m+1}$ vertices for $q$ large enough.
\end{theorem}

\begin{proof}
Let $p = \lfloor \frac{2q}{\omega} \rfloor + 1$. According to Theorem \ref{th:bipartite}, the graph $K_{p-1,(p-1)^{p-1}}$ is $p$-solvable. Then by the blow-up lemma, there exists a $q$-solvable graph with $n := \frac{q}{p} \left( (p-1)^{p-1} + p-1 \right)$ vertices and clique number $2\frac{q}{p} < \omega$. We have $n \le q (p-1)^{p-1}$, and hence for $q$ large enough
\begin{align*}
	p - 1 &\le \frac{2q}{\omega} \le 2m \frac{\log q}{\log \log q}\quad\text{and}\\
	\log n &\le \log q + 2m \frac{\log q}{\log \log q}\left\{ \log(2m) + \log \log q - \log \log \log q \right\} \le \left( 2m + 1\right) \log q,
\end{align*}
and hence $n \le q^{2m+1}$.

\end{proof}

In general, the constant $n_\epsilon$ obtained from Theorem \ref{th:bipartite} decreases rapidly with $\epsilon$. We refine it below for $\epsilon = 2/3$.

\begin{proposition} \label{prop:K22}
The complete bipartite graph $K_{2,2}$ is $3$-solvable.
\end{proposition}

\begin{proof}
Denote the bipartition as $\{v_1,v_2\} \cup \{v_3,v_4\}$. With
$$
	A = \begin{pmatrix}
	1 & 1\\
	1 & -1
	\end{pmatrix},
$$
the guessing function is given by
$$
	(f_1,f_2) = (x_3,x_4) A, \qquad (f_3,f_4) = (x_1,x_2)A^{-1}.
$$
Suppose $x$ is guessed wrong by all vertices. The vertices $v_3$ and $v_4$ guess wrong, hence we have
$$
	(x_3,x_4) = (x_1,x_2)A^{-1} + w 
$$
for some $w = (w_1,w_2) \in S:= \{(1,1),(1,2),(2,1),(2,2)\}$. Similarly, we have
$$
	(x_1,x_2) = (x_3,x_4)A + u = (x_1,x_2) + wA + u
$$
for some $u \in S$. However, it can be shown that for any $(w_1,w_2) \in S$, $wA \notin S$ and hence $wA+u \ne (0,0)$. We thus obtain the contradiction $(x_1,x_2) \ne (x_1,x_2)$.
\end{proof}

\begin{corollary} \label{cor:4q/3}
For any $q$ divisible by $3$, there exists a $q$-solvable graph on $4q/3$ vertices with clique number $2q/3$.
\end{corollary}

%
%

\section{Directed constructions} \label{sec:directed}

If we allow directed graphs, then we can further refine the constants obtained in Section \ref{sec:undirected}.

\begin{theorem} \label{th:directed}
If $q$ is even, there exists a $q$-solvable directed graph with $3q/2$ vertices and clique number $q/2$. For any $q$ divisible by 3, there exists a $q$-solvable directed graph on $4q$ vertices of clique number $q/3$. For any $q$ a multiple of four, there exists a $q$-solvable directed graph on $10q$ vertices and with clique number $q/4$.
\end{theorem}

The main strategy to produce a $p$-solvable oriented graph is by using a gadget, defined below.

\begin{definition} \label{def:gadget}
An oriented graph $D$ on $n$ vertices is called a $q$-{\em gadget} if it is not $q$-solvable, but if there exists a $D$-function $f$ over $[q]$ such that any configuration $x$ guessed incorrectly by $f$ satisfies an equality of the form $x_1 = \phi(x_2, \ldots, x_n)$ for some $\phi: [q]^{n-1} \to [q]$.
\end{definition}

\begin{lemma}[The gadget lemma] \label{lem:gadget}
If there exists a $p$-gadget on $n$ vertices, then there exists a $p$-solvable oriented graph on $n \binom{p}{2} + p$ vertices.
\end{lemma}

\begin{proof}
Start with a transitive tournament on $p$ vertices with arcs $(i,j)$ for all $i<j$. For any ordered pair $(i,j)$ with $i > j$, add a gadget $D_{i,j}$ and arcs from $i$ to all vertices in $D_{i,j}$ and whence to $j$. This yields an oriented graph $G$ on $n \binom{p}{2} + p$ vertices; we claim that $G$ is $p$-solvable.

We denote the vertices of the original tournament as $0,1,\ldots,p-1$ and for each $i > j$, the vertices of the gadget $D_{i,j}$ are $1_{i,j}, \ldots, n_{i,j}$.

Let $f$ be the function on the gadget $D$ with corresponding $\phi$. The corresponding function $g$ for $G$ is as follows:
\begin{align*}
	g_j(x) &= -\sum_{k < j} x_k - \sum_{k > j} \left[\phi(x_{2_{k,j}}, \ldots, x_{n_{k,j}}) - x_{1_{k,j}} \right] + j,\\
		g_{1_{i,j}}(x) &= f_1(x_{2_{i,j}}, \ldots, x_{n_{i,j}}) - x_i,\\
	g_{v_{i,j}}(x) &= f_v(x_{1_{i,j}} + x_i, x_{2_{i,j}}, \ldots, x_{n_{i,j}}) \qquad v=2,\ldots,n.
\end{align*}

Suppose that $x$ is guessed incorrectly by all vertices. First, all vertices in $D_{i,j}$ guess wrong; we then have
\begin{align*}
	f_1(x_{2_{i,j}}, \ldots, x_{n_{i,j}}) &\ne x_{1_{i,j}} + x_i,\\
	f_v(x_{1_{i,j}} + x_i, x_{2_{i,j}}, \ldots, x_{n_{i,j}}) &\ne x_{v_{i,j}}, \qquad v=2,\ldots,n,
\end{align*}
hence 
$$
	 x_i = \phi(x_{2_{i,j}}, \ldots, x_{n_{i,j}}) - x_{1_{i,j}}.
$$
for all $i>j$.

Now, $j$ guesses wrong, therefore
$$
	\sum_{k < j} x_k + \sum_{k > j} \left[\phi(x_{2_{k,j}}, \ldots, x_{n_{k,j}}) - x_{1_{k,j}} \right] + x_j \ne j,
$$
which combined with the above, yields
$$
	\sum_{k \in [p]} x_k \ne j.
$$
Since this holds for all $j \in [p]$, this leads to a contradiction.
\end{proof}

\begin{proposition} \label{prop:gadget}
The following graphs are gadgets.
\begin{enumerate}
	\item The graph with a single vertex and no arc is a $2$-gadget.

	\item The directed cycle on three vertices is a $3$-gadget.

	\item The graph $D$ on six vertices in Figure \ref{fig:gadget} is a $4$-gadget.
\end{enumerate}
\end{proposition}

\begin{proof}
The first graph is trivial. For the directed cycle on vertices $1,2,3$ and arcs $(1,2),(2,3),(3,1)$, the function $f$ is
\begin{align*}
	f_1(x) &= x_3\\
	f_2(x) &= x_1\\
	f_3(x) &= x_2.
\end{align*}
Therefore, $x$ is not guessed correctly by any vertex if and only if $x_1$, $x_2$, and $x_3$ are all distinct. Thus we have $\{x_1,x_2,x_3\} = [3]$ and hence $x_1 + x_2 + x_3 = 0$.

For $D$, first remark that the transpose of its adjacency matrix (i.e., the matrix $A_D$ where $A_{i,j} = 1$ if and only if $(j,i)$ is an arc in $D$) is given by
$$
	A_D = \begin{pmatrix}
	0 & 1 & 0 & 0 & 0 & 1\\
	0 & 0 & 1 & 1 & 0 & 0\\
	1 & 0 & 0 & 0 & 1 & 0\\
	1 & 0 & 1 & 0 & 0 & 1\\
	1 & 1 & 0 & 1 & 0 & 0\\
	0 & 1 & 1 & 0 & 1 & 0
	\end{pmatrix}.
$$
For ease of presentation we shall write the hat configuration $x$ as a column vector; we let $f(x) = Mx$, where 
$$
	M = \begin{pmatrix}
	0 & -1 & 0 & 0 & 0 & 1\\
	0 & 0 & -1 & 1 & 0 & 0\\
	-1 & 0 & 0 & 0 & 1 & 0\\
	-1 & 0 & 1 & 0 & 0 & 1\\
	1 & -1 & 0 & 1 & 0 & 0\\
	0 & 1 & -1 & 0 & 1 & 0
	\end{pmatrix}.
$$
Then $x$ is guessed wrong by all vertices if and only if $Lx$ is nowhere zero, where 
$$
	L = \begin{pmatrix}
	-1 & -1 & 0 & 0 & 0 & 1\\
	0 & -1 & -1 & 1 & 0 & 0\\
	-1 & 0 & -1 & 0 & 1 & 0\\
	-1 & 0 & 1 & -1 & 0 & 1\\
	1 & -1 & 0 & 1 & -1 & 0\\
	0 & 1 & -1 & 0 & 1 & -1
	\end{pmatrix}.
$$
Denoting the rows of $L$ as $L_0,\ldots,L_5$, we see that $L_3 = L_0-L_1$, $L_4 = L_1 - L_2$, $L_5 = L_2-L_0$. Therefore, $x$ is not guessed right if and only if $L_0x$, $L_1x$, and $L_2x$ are all distinct and nonzero. Therefore, $\{L_0x, L_1x, L_2x\} = \{1,2,3\}$ and $x$ must satisfy 
$$
	2x_0 + 2x_1 + 2x_2 + x_3 + x_4 + x_5 = 2
$$
Renaming the vertices such that the fifth vertex becomes first, we obtain the desired equality.
\end{proof}

However, it is still unknown whether there exist gadgets for more than four colours.

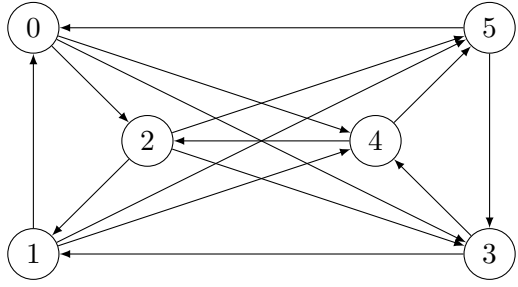
\begin{figure}[!t]
\centering
\begin{tikzpicture}
	\tikzstyle{every node}=[draw, shape = circle];
	
	\node (0) at (0,3) {0};
	\node (1) at (0,0) {1};
	\node (2) at (1.5,1.5) {2};
	\node (3) at (6,0) {3};
	\node (4) at (4.5,1.5) {4};
	\node (5) at (6,3) {5};
	
	\draw[-latex] (1) -- (0);
	\draw[-latex] (0) -- (2);
	\draw[-latex] (2) -- (1);
	
	\draw[-latex] (4) -- (5);
	\draw[-latex] (5) -- (3);
	\draw[-latex] (3) -- (4);
	
	\draw[-latex] (5) -- (0);
	\draw[-latex] (3) -- (1);
	\draw[-latex] (4) -- (2);
	
	\draw[-latex] (0) -- (3);
	\draw[-latex] (1) -- (4);
	\draw[-latex] (2) -- (5);
	
	\draw[-latex] (0) -- (4);
	\draw[-latex] (1) -- (5);
	\draw[-latex] (2) -- (3);
\end{tikzpicture}
\caption{The $4$-gadget $D$ in Proposition \ref{prop:gadget}} \label{fig:gadget}
\end{figure}

\section{Non-solvability results} \label{sec:non-solvability}



In Section~\ref{sec:undirected} we showed that a complete bipartite graph with one part of size $q-1$ was $q$-solvable.  In contrast, in this section we show that any bipartite graph that has a partition with one part of size at most $q-2$ is not $q$-solvable.  To do this we consider the following non-distinguishable set in $[q]^m$ (in other words, a subset of $[q]^m$ with remoteness $m$). Set $m = q-2$, and denote the words $w_a = (a,\ldots,a) \in [q]^m$ for all $a \in [q] \backslash \{0\}$, then $W = \{w_a: a \in [q]\backslash \{0\}\}$ is non-distinguishable. Indeed, for any $x \in [q]^m$, let $X = \{b \in [q]: x_i = b \, \mbox{for some} \, i\}$ denote the set of values taken by the coordinates of $x$, then $|X| \le m < |W|$ and hence there exists $a \in ([q] \backslash \{0\}) \backslash X$ and thus $d_H(x,w_a) = m$.

In fact, our proof applies to a larger class of graphs than bipartite graphs, defined as follows.

\begin{definition} \label{def:semibipartite}
We say a directed graph $D$ is $(m,s)$-{\em semibipartite} if its vertex set can be partitioned into $V = L \cup R$, where $|L|=m$, $|R| = s$ and $D[L]$ is an independent set and $D[R]$ is acyclic.
\end{definition}

\begin{theorem} \label{th:semibipartite}
Any $(m,s)$-semibipartite graph is not $(m+2)$-solvable.
\end{theorem}

\begin{proof}
Let $q = m+2$ and denote the vertices of $R$ as $r_1,\ldots,r_s$. Let $y \in [q]^s$ such that 
\begin{align*}
	y_1 &\notin \{f_{r_1}(w_a) : a \in [q]\}\\
	y_2 &\notin \{f_{r_2}(w_a,y_1) : a \in [q]\}\\
	&\vdots\\
	y_s &\notin \{f_{r_s}(w_a,y_1,\ldots,y_{s-1}) : a \in [q]\};
\end{align*}
such $y$ exists for each set on the right hand side has cardinality at most $|W| = q-1$. Furthermore, let $b \in [q] \backslash \{f_{l_1}(y), \ldots, f_{l_m}(y)\}$ (where $l_1,\ldots,l_m$ are the vertices of $L$), then all vertices guess $(w_b,y)$ incorrectly.
\end{proof}

This theorem is best possible, for Theorem \ref{th:bipartite} indicates that there are $q$-solvable bipartite graphs with left part of size $q-1$.

\begin{corollary} \label{cor:bipartite}
The complete bipartite graph $K_{m,n}$ is not $(m+2)$-solvable.
\end{corollary}

\begin{corollary} \label{cor:k=1}
Any graph with a minimum vertex feedback set of cardinality one is $q$-solvable if and only if $q=2$.
\end{corollary}

\begin{proof}
By Theorem \ref{th:semibipartite}, such a graph is not $3$-solvable. Conversely, it is not acyclic, hence it contains a directed cycle as a subgraph: let us prove that the directed cycle $C_n$ on $n$ vertices is $2$-solvable. Let the function be $f_1(x) = x_n$ and $f_i(x) = x_{i-1} + 1$ for $2 \le i \le n$, then $x$ is guessed incorrectly by all vertices if and only if $x_1 = x_2 = \ldots = x_n = x_1 +1$, which is clearly impossible.
\end{proof}

\begin{theorem} \label{th:bound2}
Let $D$ be a directed graph on $n$ vertices with an acyclic induced subgraph of size $I$. If 
$$
	(n-I) \left(\frac{q}{q-1}\right)^I < q,
$$
then $D$ is not $q$-solvable.
\end{theorem}

\begin{proof}
We denote the set of vertices inducing an acyclic subgraph of cardinality $I$ as $A$; we also denote a guessing function as $f$. Let $x \in [q]^I$ be the hat assignment on $A$ and $y \in [q]^{n-I}$ be the assignment on the rest of the vertices. For each choice of $y$, denote by $S_d(y)$ the set of  choices for $x$ such that exactly $d$ vertices in $A$ guess correctly for all $0 \le d \le I$. It is easy to prove by induction on $I$ that $N_d := |S_d(y)| =  \binom{I}{d}(q-1)^{I-d}$. We shall consider the situation when $x \in S_0(y)$, i.e., when no vertex in $A$ guesses correctly; given $y$, there are $N_0 = (q-1)^I$ such assignments.

For any $y$, let $G$ denote the number of times the vertices in $A$ guess their colours correctly when $x \notin S_0(y)$:
$$
	G := \sum_{x \in [q]^I} \sum_{i = 1}^I {\bf 1}\{f_{a_i}(x,y) = x_i \} = \sum_{d=1}^I d N_d = Iq^{I-1}.
$$

The total number of correct guesses, over all assignments $(x,y)$, is of course equal to $nq^{n-1}$. Therefore, there are at most 
$$
	H := nq^{n-1} - q^{n-I}G = (n-I)q^{n-1}
$$ 
correct guesses over the whole graph for any $(x,y)$ where $x \in S_0(y)$. On average, such an assignment is guessed correctly
$$
	\frac{H}{q^{n-I}N_0} = \frac{(n-I)q^{I-1}}{(q-1)^I} < 1
$$
times, and hence one hat assignment is never guessed correctly.
\end{proof}

\begin{corollary} \label{cor:bound2}
A graph with an acyclic induced subgraph of size $I$ is $q$-solvable only if it has at least $I + q \left(1 - \frac{1}{q} \right)^I$ vertices in total.
\end{corollary}


\begin{corollary} \label{cor:alpha}
If a graph on $n$ vertices has an acyclic induced subgraph of cardinality at least $n/2$, then it is $q$-solvable only if $n \ge 2\alpha (q-1)$, where $\alpha \sim 0.5675$ satisfies $\alpha + \log \alpha = 0$.
\end{corollary}

\begin{proof}
Suppose $n < 2\alpha(q-1)$, and let $i = n/(2q) < \alpha(q-1)/q$, then $\log i + i\frac{q}{q-1} < \log \frac{q-1}{q}$ and hence
\begin{align*}
	0 &> \log i + i\frac{q}{q-1}\\
	&> \log i + iq \log \left(1 + \frac{1}{q-1} \right)\\
	1 &> i \left(1 + \frac{1}{q-1} \right)^{iq}\\
	q &> \frac{n}{2} \left( \frac{q}{q-1}\right)^{\frac{n}{2}},
\end{align*}
which, by Theorem \ref{th:bound2}, shows that the graph is not $q$-solvable.
\end{proof}

\section{Even cycles} \label{sec:even_cycles}

In this section we show that a cycle whose length is a multiple of 6 is 3-solvable.  In fact, we can define guessing functions for any even cycle which have the property that at most 3 hat configurations are not guessed correctly by any vertex.

For $n>1$, let $C_{2n}$ be the cycle of length $2n$ and let $V = \{v_1,v_2,\dots,v_n\}$ and $W = \{w_1,w_2,\dots,w_n\}$ be a partition of the vertices of $C_{2n}$ into independent sets, with $v_i$ adjacent to $w_i$ and $w_{i-1}$ for all $i=1,\dots,n$ (index arithmetic taken modulo $n$).  Denote the hat colour of $v_i$ by $x_i$ and its guessing function by $f_i$.  Similarly, for $w_i$, denote its hat colour by $y_i$ and its guessing function by $g_i$.  We define the guessing functions to be
\begin{align}
f_i(y_{i-1},y_i) &=
\begin{cases}
y_i-1 & \text{if $y_i \neq y_{i-1}+1$,}\\
y_i+1 & \text{if $y_i = y_{i-1}+1$,}\\
\end{cases}\quad \text{ for $i \neq 1$};
\\
f_1(y_n,y_1) &=
\begin{cases}
y_1-1 & \text{if $y_1 \neq y_n-1$,}\\
y_1+1 & \text{if $y_1 = y_n-1$,}\\
\end{cases}
\\
g_i(x_i,x_{i+1}) &=
\begin{cases}
x_i & \text{if $x_i \neq x_{i+1}+1$,}\\
x_i-1 & \text{if $x_i = x_{i+1}+1$,}\\
\end{cases}\quad \text{ for $i \neq n$};
\\
g_n(x_n,x_1) &=
\begin{cases}
x_n & \text{if $x_n \neq x_1$,}\\
x_n-1 & \text{if $x_n = x_1$.}\\
\end{cases}
\end{align}
Graphically, we have
\begin{align*}
f_i:\: y_i\;
&\begin{tabular}{|ccc|}
\hline
1&0&1\\
2&0&0\\
2&2&1\\
\hline
\end{tabular}
&
f_1:\: y_1\;
&
\begin{tabular}{|ccc|}
\hline
0&1&1\\
0&0&2\\
2&1&2\\
\hline
\end{tabular}
\\
&\;\quad y_{i-1} & &\qquad y_n
\\
\\
g_i:\: x_{i+1}\;
&
\begin{tabular}{|ccc|}
\hline
2&1&2\\
0&1&1\\
0&0&2\\
\hline
\end{tabular}
&
g_n:\: x_1\;
&
\begin{tabular}{|ccc|}
\hline
0&1&1\\
0&0&2\\
2&1&2\\
\hline
\end{tabular}
\\
&\qquad x_i & &\qquad x_n
\end{align*}
the sets $f_i^{-1}(x_i)$ and $g_i^{-1}(y_i)$ forming L-shaped regions of $[3]^n$.
\begin{theorem}
The cycle $C_{2n}$ is 3-solvable for $n \equiv 0 \pmod 3$.
Using the guessing functions as defined above, when $n \equiv 1 \pmod 3$, the only configurations $(x,y)$ that all vertices guess incorrectly are
\[
x = (a,a+2,a+1,a,\dots,a), y = (a,a+2,a+1,a,\dots,a) \text{ for some $a \in [3]$},
\]
and when $n \equiv 2 \pmod 3$, the only configurations $(x,y)$ that all vertices guess incorrectly are
\[
x = (a+2,a,a+1,a+2,\dots,a), y = (a,a+1,a+2,a,\dots,a+1) \text{ for some $a \in [3]$}.
\]
\end{theorem}

\begin{proof}
Suppose $y = (y_1,\dots,y_n) \in [3]^n$ is the configuration of hat colours for the vertices in $W$ and that each vertex in $W$ guesses incorrectly.  Then $x \in \bigcap_{i=1}^n g_i^{-1}(y_i)^{\rm c}$, where
\[
\begin{split}
\bigcap_{i=1}^n g_i^{-1}(y_i)^{\rm c} = \bigcap_{i<n} &\{ x: x_i = y_i-1 \text{ or } x_{i+1} = y_i-1 \text{ or } (x_i,x_{i+1}) = (y_i+1,y_i+1) \}\\
&\cap \{x : x_n = y_n-1 \text{ or } x_1 = y_n \text{ or } (x_n,x_1) = (y_n+1,y_n-1) \}
\end{split}
\]

Suppose further that each vertex in $V$ guesses incorrectly.  We claim the following implications are true.
\begin{claim*} If $(x,y)$ is guessed incorrectly by all vertices then the following hold.  For all $i \neq 1$,
\begin{itemize}
\item $A_i$: if $x_i = y_i-1$ then $y_i = y_{i-1}+1$ and $x_{i-1} = y_{i-1}-1$;
\item $B_i$: if $x_i = y_i+1$ then either
\begin{enumerate}
\item $y_i = y_{i-1}$ and $x_{i-1} = y_{i-1}+1$, or
\item $y_i \neq y_{i-1}+1$ and $x_{i-1} = y_{i-1}-1$;
\end{enumerate}
\end{itemize}
and for all $i \neq n$,
\begin{itemize}
\item $C_i$: if $x_i = y_i$ then $y_{i+1} = y_i - 1$ and $x_{i+1} = y_{i+1}$.
\end{itemize}
\end{claim*}

\begin{proof}[Proof of Claim]
Take $i \neq 1$, and suppose $x_i = y_i-1$.  Since $v_i$ guesses incorrectly, we must have $y_i = y_{i-1}+1$, so that $x_i = y_{i-1}$.  But $x \in g_{i-1}^{-1}(y_{i-1})^{\rm c}$ which implies that $x_{i-1} = y_{i-1}-1$, establishing $A_i$.

Now suppose $x_i = y_i+1$.  Since $v_i$ guesses incorrectly, we must have $y_i \neq y_{i-1}+1$, so that $x_i \neq y_{i-1}-1$.  But $x \in g_{i-1}^{-1}(y_{i-1})^{\rm c}$ which implies that either $x_{i-1} = y_{i-1}-1$ or $(x_{i-1},x_i) = (y_{i-1}+1,y_{i-1}+1)$, the latter implying that $y_i = y_{i-1}$, which establishes $B_i$.

Finally, take $i \neq n$ and suppose $x_i = y_i$.  Since $x \in g_i^{-1}(y_i)^{\rm c}$, we must have $x_{i+1} = y_i-1$.  But $v_{i+1}$ guesses incorrectly which implies that $y_{i+1} = y_i - 1$.  To see this, use the fact that the function $f_i$, for $i \neq 1$, can also be written as
\[
f_i(y_{i-1},y_i) =
\begin{cases}
y_{i-1}-1 & \text{if $y_i \neq y_{i-1} - 1$},\\
y_{i-1}+1 & \text{if $y_i = y_{i-1} - 1$}.
\end{cases}
\]
Therefore $x_{i+1} = y_{i+1}$, establishing $C_i$.
\end{proof}

We use the implications $A_i,B_i$ and $C_i$ as follows.  First, suppose $x_n = y_n-1$. Then using the chain of implications $A_n,A_{n-1},\dots,A_2$ we find that $x_i = y_i-1$ for all $i$ and $y_i = y_{i-1} + 1$ for $i\neq 1$, so $y_n = y_1 + (n-1) \pmod 3$.  Since $x_1 = y_1-1$ and $v_1$ also guesses incorrectly, we must have $y_1 = y_n - 1$, a contradiction unless $n \equiv 2 \pmod 3$.  When $n \equiv 2 \pmod 3$, we discover that the configurations $x = (a+2,a,a+1,a+2,\dots,a), y = (a,a+1,a+2,a,\dots,a+1)$ for $a \in [3]$ are guessed incorrectly by all vertices.

Now suppose $x_n = y_n+1$.  Since $x \in g_n^{-1}(y_n)^{\rm c}$ we have that $x_1 \neq y_n+1$.  We consider the chain of implications $B_n,B_{n-1},\dots$ for as far as possible and note that case 1 of $B_i$ cannot occur for all $i \neq 1$, for then $x_i = y_i + 1$ for all $i$ and $y_i = y_{i-1}$ for $i \neq 1$, contradicting the fact that $x_1 \neq y_n+1$.  This means that for some $k > 1$ case 2 of $B_k$ occurs, so that $x_{k-1} = y_{k-1} - 1$.  We then apply the chain of implications $A_{k-1},A_{k-2},\dots,A_2$ to find that $x_i = y_i -1 $ for all $i < k$, so in particular $x_1 = y_1 - 1$.  Since $v_1$ guesses incorrectly, we have that $y_1 = y_n - 1$ which contradicts the fact that $x_1 \neq y_n+1$.  Hence for any configuration with $x_n = y_n+1$ there must be some vertex that guesses correctly.

Finally, suppose $x_n = y_n$.  Since $x \in g_n^{-1}(y_n)^{\rm c}$  we have that $x_1 = y_n$.  Since $v_1$ guesses incorrectly, we must have that $y_1 = y_n$.  To see this use the fact that $f_1$ can be also written as
\[
f_1(y_n,y_1) =
\begin{cases}
y_n & \text{if $y_1 \neq y_n$},\\
y_n-1 & \text{if $y_1 = y_n$}.
\end{cases}
\]
Therefore $x_1 = y_1$.  We now apply the chain of implications $C_1,C_2,\dots,C_{n-1}$ to find that $x_i = y_i$ for all $i$, and $y_{i+1} = y_i - 1$ for $i \neq n$.  Therefore $y_n = y_1 - (n-1) \pmod 3$, which is a contradiction unless $n \equiv 1 \pmod 3$.  When $n \equiv 1 \pmod 3$, we discover that the configurations $x=(a,a+2,a+1,a,\dotsc,a), y=(a,a+2,a+1,a,\dots,a)$ for $a \in [3]$ are guessed incorrectly by all vertices.
\end{proof}

Unfortunately, this `L-shaped' construction falls just short of proving 3-solvability when $n \not\equiv 0 \pmod 3$; indeed, out of the $3^{2n}$ possible hat configurations, there are only 3 where all vertices guess incorrectly!  


In any case, the family of cycles of length a multiple of 6 gives an answer to a question of Butler et al.\@ about edge-critical graphs.  A graph $G$ is called \emph{edge-critical for $q$ colours} if $G$ is $q$-solvable, but $G - e$ is not $q$-solvable for any edge $e \in G$.  For $q=2$ the only edge-critial graph is the graph of a single edge, and for $q>2$ there are at least two distinct edge-critical graphs, namely $K_q$ and some subgraph of the bipartite graph $K_{q-1,(q-1)^{q-1}}$ presented earlier.  Butler et al.\@ ask whether there are infinitely many graphs which are edge-critical for $q$ colours, for $q>2$.  Since trees are known not to be 3-solvable, the cycles of length a multiple of 6 form such an infinite family for $q=3$.

\begin{theorem} 
The family $\{ C_{6k} : k \in \mathbb{N} \}$ is an infinite family of edge-critical graphs for 3~colours.
\end{theorem}

\section{Acknowledgment}
This work was produced while the second author was supported by the Engineering and Physical Sciences Research Council [grant number EP/J021784/1].



\end{document}